\documentclass[preprint,12pt]{elsarticle}

\usepackage{amsmath,amssymb,amsthm}
\usepackage{mathtools}
\usepackage[hidelinks]{hyperref}

\journal{Discrete Applied Mathematics}

\newtheorem{theorem}{Theorem}[section]
\newtheorem{lemma}[theorem]{Lemma}

\newtheorem{claim}{Claim}[section]

\begin{document}

\begin{frontmatter}

\title{Every \(2\)-Subdivision of a Cubic Graph Is Antimagic}

\author[ndhu]{Fei-Huang Chang\corref{cor1}\fnref{fund1}}
\ead{cfh@gms.ndhu.edu.tw}

\author[ntu]{Teng-Da Chang}
\ead{s10001218@gmail.com}

\author[tku]{Zhishi Pan\fnref{fund2}}
\ead{zhishi@mail.tku.edu.tw}

\cortext[cor1]{Corresponding author.}

\address[ndhu]{Department of Applied Mathematics,
National Dong Hwa University, Hualien 974301, Taiwan}

\address[ntu]{Department of Mathematics,
National Taiwan University, Taipei 10617, Taiwan}

\address[tku]{Department of Mathematics,
Tamkang University, New Taipei City 251301, Taiwan}

\fntext[fund1]{Fei-Huang Chang was supported by the National Science and
Technology Council of Taiwan under Grant NSTC 114-2115-M-259-005.}

\fntext[fund2]{Zhishi Pan was supported by the National Science and
Technology Council of Taiwan under Grant NSTC 114-2115-M-032-004.}

% ...
\begin{abstract}
Let \(G\) be a finite simple cubic graph, not necessarily connected, and
let \(S_2(G)\) be obtained by subdividing every edge of \(G\) twice.
Li(2025) developed general constructions for antimagic labelings of repeated
subdivisions, but the cubic case \(G(3)=S_2(G)\) is not covered by those
methods. We prove that \(S_2(G)\) is strongly antimagic. Our first proof
constructs an edge labeling of \(G\) in which every vertex sum is
sufficiently large and occurs at most twice, and then uses an orientation
after subdivision to separate the remaining equal sums. A second, direct
construction uses the same path decomposition to make the internal
contribution at each original vertex constant, while a unique endpoint
contribution distinguishes the resulting sums. The direct construction
further shows that \(S_2(G)\) is strongly antimagic whenever every vertex
of \(G\) has odd degree at least three.
\end{abstract}


\end{frontmatter}

\section{Introduction}
\label{sec:introduction}

All graphs considered in this paper are finite and simple. For a graph
\(G\) and a vertex \(v\in V(G)\), we write \(d_G(v)\) for the degree of
\(v\) in \(G\). Unless explicitly stated otherwise, the graphs are not
required to be connected. Let \(G\) be a graph with \(m\) edges. An
\emph{antimagic labeling} of \(G\) is a bijection
\[
f:E(G)\longrightarrow [m]
\]
such that the induced vertex sums
\[
\sigma_f(v)=\sum_{e\ni v}f(e),
\qquad v\in V(G),
\]
are pairwise distinct. Hartsfield and Ringel introduced this notion and
conjectured that every connected graph other than \(K_2\) is antimagic
\cite{hartsfield1989supermagic}. Although the conjecture remains open in
full generality, it has been verified for many important graph classes.

An antimagic labeling \(f\) is called \emph{strongly antimagic} if
\[
d_G(u)<d_G(v)
\quad\Longrightarrow\quad
\sigma_f(u)<\sigma_f(v)
\]
for every pair of vertices \(u,v\in V(G)\). Thus, a strongly antimagic
labeling not only distinguishes all vertex sums, but also orders them
according to vertex degree. This stronger property has been studied for
several families of nonregular graphs. In particular, double spiders were
shown to be strongly antimagic by Chang, Chin, Li, and Pan
\cite{chang2020double}; a later inductive treatment provided a shorter
proof and corrected a subtle point in the original argument
\cite{liu2022inductive}.

For a positive integer \(\ell\), the \(\ell\)-subdivision \(S_\ell(G)\)
is obtained by subdividing every edge of \(G\) exactly \(\ell\) times.
Thus, every original edge is replaced by a path of length \(\ell+1\). In
particular, \(S_2(G)\) is obtained by replacing every edge of \(G\) with
a path of length three. This convention differs from that of Li
\cite{li2025subdivisions}, who writes \(G(s)\) for the graph obtained by
replacing every edge with a path of length \(s\). Consequently,
\[
S_\ell(G)=G(\ell+1),
\]
and the graph considered here is \(S_2(G)=G(3)\).

Regular graphs form one of the best understood families in antimagic
labeling theory. Cranston proved that regular bipartite graphs of degree
at least two are antimagic \cite{cranston2009regular}. Liang and Zhu
established the cubic case \cite{liang2014cubic}; Cranston, Liang, and Zhu
treated regular graphs of odd degree \cite{cranston2015odd}; and Chang,
Liang, Pan, and Zhu completed the even-degree case
\cite{chang2016regular}. Hence every regular graph of degree at least two
is known to be antimagic. These results concern the original regular
graph, however, and do not directly determine whether its subdivisions
are antimagic. Subdivision introduces many vertices of degree two, whose
sums must be separated from one another and from the sums at the original
vertices.

The difficulty caused by vertices of degree two is already visible in the
study of trees. Liang, Wong, and Zhu obtained several sufficient
conditions for trees containing degree-two vertices to be antimagic
\cite{liang2014trees}. Their work illustrates that subdivision is not
merely a formal graph operation in the antimagic setting: the labels along
the replacing paths must be coordinated carefully in order to control the
sums at the newly inserted vertices.

A closely related result was obtained by Deng and Li
\cite{deng2022biregular}. They proved that every connected
\((2,k)\)-biregular bipartite graph with \(k\ge3\) is antimagic. If \(G\)
is a connected \(k\)-regular graph and every edge of \(G\) is subdivided
exactly once, then the resulting graph is \((k,2)\)-biregular. Their
theorem therefore implies that \(S_1(G)\) is antimagic for every connected
\(k\)-regular graph with \(k\ge3\). This settles the case in which every
original edge is replaced by a path of length two, but it does not address
repeated subdivisions, where consecutive degree-two vertices occur on
each replacing path.

Li developed a systematic framework for antimagic labelings of repeated
subdivisions \cite{li2025subdivisions}. For a graph \(G\) and an integer
\(s\ge2\), the paper studies the graph \(G(s)\) obtained by replacing
every edge of \(G\) with a path of length \(s\), and gives constructive
sufficient conditions under which \(G(s)\) is antimagic. The method for
regular graphs does not apply to degrees three and five. In particular,
it does not cover the cubic graph \(G(3)\), which is precisely \(S_2(G)\).
Thus, the present problem is a natural low-degree and short-subdivision
case outside the previously available general constructions.

The main result of this paper is the following.

\medskip
\noindent\textbf{Main result.}
For every finite simple cubic graph \(G\), not necessarily connected, the
graph \(S_2(G)\) admits a strongly antimagic labeling.
\medskip

As a further consequence of the direct construction, we extend the result
to every finite simple graph whose vertices all have odd degree at least
three.

When \(G\) is cubic, \(S_2(G)\) has only vertices of degrees two and
three. Thus, in the cubic case, the additional strongly antimagic
requirement means that every original degree-three vertex must receive a
larger sum than every newly inserted degree-two vertex. Both constructions
in this paper satisfy this condition: all subdivision-vertex sums are at
most \(4m\), whereas all original-vertex sums are greater than \(4m\).

We give two proofs. In the first, we decompose the edges of the cubic
graph into paths and label the original edges so that all vertex sums are
sufficiently large and no sum occurs more than twice. Vertices receiving
the same sum are paired, and an orientation obtained from an Euler tour of
an auxiliary \(4\)-regular multigraph assigns different outdegrees to the
two vertices in each pair. The labeling is then transferred to \(S_2(G)\),
where the outdegree correction separates the remaining equal sums at the
original vertices.

The proof of the first construction suggests a more direct approach. We
again orient the paths in the decomposition and assign alternating large
and small indices to their edges. This time, however, the three labels on
each subdivided edge are ordered according to whether its index comes from
the large or small part of the alternating sequence. As a result, the two
subdivided edges corresponding to the path on which an original vertex is
internal contribute the same constant at every original vertex. The label
arising from the unique path ending at that vertex then distinguishes its
total sum from all others. This gives a direct strongly antimagic labeling
of \(S_2(G)\).

The path decomposition used in the two constructions is valid more
generally for graphs in which every vertex has odd degree. By using the
full strength of Claim~\ref{clm:path-decomposition} together with the
direct labeling developed in Section~\ref{sec:direct-labeling}, we further
prove that \(S_2(G)\) is strongly antimagic whenever every vertex of \(G\)
has odd degree at least three. This extension is presented in
Section~\ref{sec:odd-degree-extension}.

The first proof is retained because it separates the argument into steps
that may be useful in other subdivision problems: one first allows a
controlled amount of repetition among the sums on the original graph and
then removes those repetitions after subdivision. The second proof shows
how the particular structure revealed by that argument can be
incorporated directly into the labeling of \(S_2(G)\).

The remainder of the paper is organized as follows. In
Section~\ref{sec:main-results}, we give the first proof for cubic graphs by
combining an auxiliary edge labeling with a suitable orientation. In
Section~\ref{sec:direct-labeling}, we present a direct strongly antimagic
labeling of \(S_2(G)\). In Section~\ref{sec:odd-degree-extension}, we
extend the direct construction to graphs in which every vertex has odd
degree at least three. Section~\ref{sec:conclusion} concludes the paper
and discusses several related open problems.

\section{Antimagic Labelings of \(2\)-Subdivisions of Cubic Graphs}
\label{sec:main-results}

Throughout this section, let \(G\) be a finite simple cubic graph, not
necessarily connected, of order \(n\) and size \(m=3n/2\). For a positive
integer \(r\), write
\[
[r]=\{1,2,\ldots,r\}.
\]
If \(g:E(G)\to[m]\) is an edge labeling, let
\[
\sigma_g(v)=\sum_{e\ni v}g(e)
\]
denote the sum induced at \(v\).

\subsection{An auxiliary labeling of the cubic graph}
\label{subsec:auxiliary-labeling}

\begin{lemma}
\label{lem:auxiliary-labeling}
There exists a bijection \(g:E(G)\to[m]\) such that
\[
\sigma_g(v)\ge m+1
\qquad\text{for every }v\in V(G),
\]
and every induced vertex sum occurs at most twice.
\end{lemma}

We first record a path decomposition that applies more generally to graphs
in which every vertex has odd degree. Its specialization to cubic graphs
will be used throughout this section, while its full strength will be used
in Section~\ref{sec:odd-degree-extension}.

\begin{claim}
\label{clm:path-decomposition}
Let \(G\) be a finite simple graph, not necessarily connected, of order
\(n\), and suppose that every vertex of \(G\) has odd degree. Then \(E(G)\)
can be decomposed into \(n/2\) paths such that every vertex \(v\in V(G)\)
is an end vertex of exactly one path and an internal vertex of exactly
\[
\frac{d_G(v)-1}{2}
\]
paths.
\end{claim}

\begin{proof}
Apply the path-decomposition theorem of Lov\'asz
\cite{lovasz1968covering} to each connected component of \(G\). Since
every vertex has odd degree, the resulting decompositions combine to give
a path decomposition \(\mathcal P\) of \(G\) satisfying
\[
|\mathcal P|\le \frac{n}{2}.
\]

For each \(v\in V(G)\), let \(r(v)\) denote the number of paths in
\(\mathcal P\) having \(v\) as an end vertex, and let \(q(v)\) denote the
number of paths in which \(v\) is an internal vertex. Since the paths in
\(\mathcal P\) are edge-disjoint and cover \(E(G)\),
\[
d_G(v)=2q(v)+r(v).
\]
Because \(d_G(v)\) is odd, \(r(v)\) is odd. In particular, \(r(v)\ge1\)
for every \(v\in V(G)\). Counting all path ends gives
\[
2|\mathcal P|
=
\sum_{v\in V(G)}r(v)
\ge n.
\]
Together with \(|\mathcal P|\le n/2\), this yields
\[
|\mathcal P|=\frac{n}{2}
\qquad\text{and}\qquad
\sum_{v\in V(G)}r(v)=n.
\]
Since every \(r(v)\) is a positive integer, equality implies
\[
r(v)=1
\qquad\text{for every }v\in V(G).
\]
Substituting this into \(d_G(v)=2q(v)+r(v)\), we obtain
\[
q(v)=\frac{d_G(v)-1}{2}.
\]
Thus every vertex is an end vertex of exactly one path and an internal
vertex of exactly \((d_G(v)-1)/2\) paths.
\end{proof}

Since \(G\) is cubic throughout the remainder of this section,
Claim~\ref{clm:path-decomposition} implies that every vertex of \(G\) is
an end vertex of exactly one path and an internal vertex of exactly one
path.

\begin{proof}[Proof of Lemma~\ref{lem:auxiliary-labeling}]
Fix a path decomposition \(\mathcal P\) as in
Claim~\ref{clm:path-decomposition}. Its \(n/2\) paths have total length
\(m=3n/2\), so at least one of them has length at least three. Let \(t\)
be the number of paths of length one, and denote the remaining paths by
\[
Q_1,Q_2,\ldots,Q_s,
\]
where \(Q_1\) has length at least three. Give every \(Q_i\) an arbitrary
direction, list its edges in that direction, and concatenate the lists to
obtain
\[
e_1,e_2,\ldots,e_{m-t}.
\]

Reserve the labels in \([t]\) for the paths of length one. On the
remaining edges, define
\[
g(e_{2j-1})=m-j+1
\qquad\text{and}\qquad
g(e_{2j})=t+j
\]
whenever the indicated edge exists. Thus the labels are assigned in the
global order
\[
m,\ t+1,\ m-1,\ t+2,\ m-2,\ t+3,\ldots,
\]
without restarting the alternation when one path ends and the next
begins.

For \(v\in V(G)\), let \(\rho(v)\) be the sum of the labels on the two
edges of the path on which \(v\) is internal.

\begin{claim}
\label{clm:internal-pair-sums}
For every \(v\in V(G)\),
\[
\rho(v)\in\{m+t,m+t+1\}.
\]
Moreover, the label \(t+1\) is not assigned to an end edge of any
\(Q_i\).
\end{claim}

\begin{proof}
Every two consecutive labels in the global sequence have sum
\[
(m-j+1)+(t+j)=m+t+1
\]
or
\[
(t+j)+(m-j)=m+t.
\]
Since \(G\) is cubic, \(v\) is internal to a unique path of
\(\mathcal P\). The two edges of this path incident with \(v\) are
consecutive in the global sequence. The label \(t+1\) is assigned to
\(e_2\); since \(Q_1\) has length at least three, \(e_2\) is not an end
edge.
\end{proof}

Let \(V_0\) be the set of vertices whose unique path in \(\mathcal P\)
ending at that vertex has length one, and put
\[
V_1=V(G)\setminus V_0.
\]
For \(v\in V_1\), let \(c(v)\) be the label on the end edge of the
nontrivial path ending at \(v\).

\begin{claim}
\label{clm:V1-sums}
Every vertex sum induced on \(V_1\) occurs at most twice, and
\[
\sigma_g(v)\ge m+2t+2
\qquad\text{for every }v\in V_1.
\]
\end{claim}

\begin{proof}
By Claim~\ref{clm:internal-pair-sums}, the label \(t+1\) is not assigned
to an end edge of a nontrivial path, so \(c(v)\ge t+2\). Moreover, the
values \(c(v)\), \(v\in V_1\), are pairwise distinct, because distinct
endpoint occurrences correspond to distinct end edges.

Write
\[
\rho(v)=m+t+\varepsilon(v),
\qquad
\varepsilon(v)\in\{0,1\}.
\]
Then
\[
\sigma_g(v)=m+t+\varepsilon(v)+c(v).
\]
For a fixed vertex sum, there is at most one possible vertex for each of
the two values of \(\varepsilon(v)\), since the values \(c(v)\) are
pairwise distinct. Hence every sum on \(V_1\) occurs at most twice.
Finally,
\[
\sigma_g(v)\ge(m+t)+(t+2)=m+2t+2.
\]
\end{proof}

It remains to assign the reserved labels in \([t]\) to the paths of
length one.

\begin{claim}
\label{clm:V0-sums}
The labels in \([t]\) can be assigned to the paths of length one so that
every vertex sum induced on \(V_0\) occurs at most twice. Moreover,
\[
\sigma_g(v)\le m+2t+1
\qquad\text{for every }v\in V_0.
\]
\end{claim}

\begin{proof}
The assertion is vacuous when \(t=0\), so assume \(t\ge1\). By
Claim~\ref{clm:path-decomposition}, the \(t\) paths of length one are
pairwise vertex-disjoint, and their \(2t\) end vertices are precisely the
vertices in \(V_0\).

Put \(A=m+t\). For a single-edge path \(e=uv\), define
\[
a(e)
=
\left|
\{x\in\{u,v\}:\rho(x)=A\}
\right|.
\]
Thus \(a(e)\in\{0,1,2\}\). Order the single-edge paths as
\[
e^{(1)},e^{(2)},\ldots,e^{(t)}
\]
so that
\[
a\bigl(e^{(1)}\bigr)
\ge
a\bigl(e^{(2)}\bigr)
\ge\cdots\ge
a\bigl(e^{(t)}\bigr),
\]
and assign
\[
g\bigl(e^{(j)}\bigr)=j.
\]
Write
\[
a_j=a\bigl(e^{(j)}\bigr).
\]

If \(x\) is an end vertex of \(e^{(j)}\), then
\[
\sigma_g(x)\in\{A+j,A+j+1\}.
\]
The multiplicity of \(A+1\) is \(a_1\le2\). For \(2\le k\le t\), the sum
\(A+k\) can arise only from an end of \(e^{(k)}\) with partial sum \(A\),
or from an end of \(e^{(k-1)}\) with partial sum \(A+1\). Its
multiplicity is therefore
\[
a_k+(2-a_{k-1})\le2,
\]
where the inequality follows from \(a_k\le a_{k-1}\). Finally, the
multiplicity of \(A+t+1\) is \(2-a_t\le2\). Thus every sum on \(V_0\)
occurs at most twice, and
\[
\sigma_g(x)\le A+t+1=m+2t+1
\]
for every \(x\in V_0\).
\end{proof}

The alternating assignment uses every label in
\(\{t+1,\ldots,m\}\) exactly once. Indeed, if \(m-t=2q\), the low and
high parts are
\[
\{t+1,\ldots,t+q\}
\qquad\text{and}\qquad
\{m-q+1,\ldots,m\},
\]
and
\[
m-q+1=t+q+1.
\]
If \(m-t=2q+1\), they are
\[
\{t+1,\ldots,t+q\}
\qquad\text{and}\qquad
\{m-q,\ldots,m\},
\]
and
\[
m-q=t+q+1.
\]
Together with the labels in \([t]\), this proves that
\(g:E(G)\to[m]\) is a bijection.

By definition,
\[
V(G)=V_0\mathbin{\dot\cup}V_1.
\]
Claims~\ref{clm:V1-sums} and~\ref{clm:V0-sums} show that every induced
sum occurs at most twice, and their respective ranges are disjoint.
Finally,
\[
\sigma_g(v)\ge\rho(v)+1\ge m+t+1\ge m+1
\]
for every \(v\in V(G)\). This completes the proof.
\end{proof}

\subsection{Separating the sums at the original vertices}
\label{subsec:separating-original-sums}

\begin{lemma}
\label{lem:orientation-transfer}
Let \(g:E(G)\to[m]\) be a bijection such that every vertex sum induced by
\(g\) occurs at most twice. Then there exist an orientation \(D\) of \(G\)
and a bijection
\[
f:E(S_2(G))\longrightarrow[3m]
\]
such that the sums induced by \(f\) at the original vertices of \(G\) are
pairwise distinct.
\end{lemma}

\begin{proof}
For every vertex sum that occurs twice, pair the two vertices receiving
that sum. Pair the remaining vertices arbitrarily; this is possible
because a cubic graph has even order. Let \(M\) be the resulting
auxiliary perfect matching on \(V(G)\).

\begin{claim}
\label{clm:orientation}
The graph \(G\) has an orientation \(D\) such that
\[
d_D^+(v)\in\{1,2\}
\qquad\text{for every }v\in V(G),
\]
and any two distinct vertices with the same \(g\)-sum have different
outdegrees in \(D\).
\end{claim}

\begin{proof}
The multigraph \(G+M\) is \(4\)-regular. Orient an Euler circuit in each
component and then delete the edges of \(M\). Every vertex has outdegree
\(1\) or \(2\). Moreover, if an edge \(uv\in M\) is oriented from \(u\)
to \(v\), then deleting it leaves
\[
d_D^+(u)=1
\qquad\text{and}\qquad
d_D^+(v)=2.
\]
Hence the two vertices in every prescribed equal-sum pair have different
outdegrees.
\end{proof}

Write
\[
E(G)=\{e_1,e_2,\ldots,e_m\}
\]
so that \(g(e_i)=i\), and orient \(e_i\) from \(u_i\) to \(v_i\)
according to \(D\). In \(S_2(G)\), replace \(e_i\) by the path
\[
u_i x_i y_i v_i
\]
and define
\[
f(x_i y_i)=i,\qquad
f(u_i x_i)=m+2i-1,\qquad
f(y_i v_i)=m+2i.
\]
As \(i\) ranges over \([m]\), the labels \(i\), \(m+2i-1\), and \(m+2i\)
form exactly the set \([3m]\), so \(f\) is a bijection.

For an original vertex \(v\), an edge \(e_i\) directed out of \(v\)
contributes \(m+2i-1\), whereas an edge directed into \(v\) contributes
\(m+2i\). Since \(G\) is cubic,
\[
\sigma_f(v)=3m+2\sigma_g(v)-d_D^+(v).
\]
If two original vertices \(u\) and \(v\) satisfy
\(\sigma_f(u)=\sigma_f(v)\), then
\[
2\bigl(\sigma_g(u)-\sigma_g(v)\bigr)
=
d_D^+(u)-d_D^+(v).
\]
The left-hand side is even, while the right-hand side belongs to
\(\{-1,0,1\}\). Thus both sides are zero. If \(u\ne v\), then \(u\) and
\(v\) are the prescribed pair corresponding to their common \(g\)-sum,
contradicting Claim~\ref{clm:orientation}. Hence \(u=v\), and the
original-vertex sums are pairwise distinct.
\end{proof}

\subsection{Proof of the main theorem}
\label{subsec:proof-main-theorem}

\begin{theorem}
\label{thm:main}
The \(2\)-subdivision of every finite simple cubic graph is strongly
antimagic, regardless of whether the graph is connected.
\end{theorem}

\begin{proof}
Let \(g:E(G)\to[m]\) be the labeling supplied by
Lemma~\ref{lem:auxiliary-labeling}. Apply
Lemma~\ref{lem:orientation-transfer} to obtain an orientation \(D\) and
the corresponding bijection
\[
f:E(S_2(G))\longrightarrow[3m].
\]
The sums induced by \(f\) at the original vertices are pairwise distinct.

\begin{claim}
\label{clm:subdivision-vertex-sums}
All sums induced by \(f\) at the subdivision vertices are pairwise
distinct, and each of them is at most \(4m\).
\end{claim}

\begin{proof}
For every \(i\in[m]\),
\[
\sigma_f(x_i)=m+3i-1
\qquad\text{and}\qquad
\sigma_f(y_i)=m+3i.
\]
Each sequence is strictly increasing, and an equality between the two
sequences would imply
\[
3(i-j)=1,
\]
which is impossible. Thus all subdivision-vertex sums are pairwise
distinct, and their maximum is
\[
\sigma_f(y_m)=4m.
\]
\end{proof}

For every original vertex \(v\), Lemma~\ref{lem:auxiliary-labeling} and
\(d_D^+(v)\le2\) give
\[
\begin{aligned}
\sigma_f(v)
&=3m+2\sigma_g(v)-d_D^+(v)\\
&\ge3m+2(m+1)-2\\
&=5m\\
&>4m.
\end{aligned}
\]
Therefore the original-vertex sums are pairwise distinct, the
subdivision-vertex sums are pairwise distinct, and every degree-three
vertex has larger sum than every degree-two vertex. Hence \(f\) is a
strongly antimagic labeling of \(S_2(G)\).
\end{proof}

\section{A Direct Antimagic Labeling}
\label{sec:direct-labeling}

The proof in the preceding section proceeds in two stages. It first
constructs a labeling of the original cubic graph in which every vertex
sum is sufficiently large and occurs at most twice, and then uses an
orientation to separate the remaining equal sums after subdivision. The
structure of that argument also suggests a more direct construction on
\(S_2(G)\).

Using the same path decomposition, we orient each path and assign
alternating large and small indices to its edges. We then order the two
outer labels on each subdivided edge according to the type of its index.
This arrangement makes the contribution from the two edges belonging to
the path on which an original vertex is internal the same for every
original vertex. The remaining contribution, arising from the unique path
ending at that vertex, then distinguishes its total sum from all others.

We now give this construction as an alternative proof of
Theorem~\ref{thm:main}.

\begin{proof}[Alternative proof of Theorem~\ref{thm:main}]
Let \(G\) be a finite simple cubic graph of order \(n\) and size
\(m=3n/2\). By Claim~\ref{clm:path-decomposition}, the edge set of \(G\)
admits a decomposition \(\mathcal P\) into \(n/2\) paths such that every
vertex is an end vertex of exactly one path. Since \(G\) is cubic, every
vertex is also an internal vertex of exactly one path.

Let \(t\) be the number of paths of length one. Denote the nontrivial
paths by
\[
Q_1,Q_2,\ldots,Q_s,
\]
where \(Q_1\) has length at least three. Give each \(Q_i\) an arbitrary
direction, list its edges in that direction, and concatenate the lists to
obtain
\[
e_1,e_2,\ldots,e_{m-t}.
\]

As in the proof of Lemma~\ref{lem:auxiliary-labeling}, assign an index
\(\iota(e)\in[m]\) to every edge of \(G\) as follows. Assign the indices
\(1,2,\ldots,t\) bijectively to the paths of length one, and set
\[
\iota(e_{2j-1})=m-j+1
\qquad\text{and}\qquad
\iota(e_{2j})=t+j
\]
whenever the indicated edge exists. Thus, along the directed nontrivial
paths, the indices occur in the global order
\[
m,\ t+1,\ m-1,\ t+2,\ m-2,\ t+3,\ldots.
\]
The verification in the proof of Lemma~\ref{lem:auxiliary-labeling}
shows that
\[
\iota:E(G)\longrightarrow[m]
\]
is a bijection.

Call an edge \(e_{2j-1}\) a \emph{large-type edge} and an edge \(e_{2j}\)
a \emph{small-type edge}. Orient each path of length one arbitrarily. For
an oriented edge \(e=uv\) with \(\iota(e)=i\), write the corresponding
path in \(S_2(G)\) as
\[
u x_e y_e v.
\]
Define
\[
f:E(S_2(G))\longrightarrow[3m]
\]
as follows. If \(e\) is a large-type edge, or if \(e\) is a path of
length one, assign
\[
f(ux_e)=m+2i,\qquad
f(x_e y_e)=i,\qquad
f(y_e v)=m+2i-1.
\]
If \(e\) is a small-type edge, assign
\[
f(ux_e)=m+2i-1,\qquad
f(x_e y_e)=i,\qquad
f(y_e v)=m+2i.
\]
For each \(i\in[m]\), the three labels associated with the edge of index
\(i\) are precisely
\[
i,\qquad m+2i-1,\qquad m+2i.
\]
Hence \(f\) is a bijection from \(E(S_2(G))\) to \([3m]\).

The order of the two outer labels does not affect the sums at the
subdivision vertices. As in Claim~\ref{clm:subdivision-vertex-sums}, the
two subdivision vertices on the path corresponding to index \(i\) receive
the sums
\[
m+3i-1
\qquad\text{and}\qquad
m+3i.
\]
Consequently, all subdivision-vertex sums are pairwise distinct and are
at most \(4m\).

We now consider an original vertex \(v\). By
Claim~\ref{clm:path-decomposition}, \(v\) is an internal vertex of
exactly one directed path in \(\mathcal P\). The two edges of that path
incident with \(v\) are consecutive in the global list.

Suppose first that a large-type edge of index \(m-j+1\) enters \(v\) and
a small-type edge of index \(t+j\) leaves \(v\). Their contributions at
\(v\) are
\[
m+2(m-j+1)-1
\qquad\text{and}\qquad
m+2(t+j)-1,
\]
respectively, and their sum is
\[
4m+2t.
\]

Suppose instead that a small-type edge of index \(t+j\) enters \(v\) and
a large-type edge of index \(m-j\) leaves \(v\). Their contributions at
\(v\) are
\[
m+2(t+j)
\qquad\text{and}\qquad
m+2(m-j),
\]
respectively, and their sum is again
\[
4m+2t.
\]
Thus, at every original vertex, the two edges belonging to the path on
which the vertex is internal contribute the same constant \(4m+2t\).

Let \(\lambda(v)\) be the label on the remaining edge incident with
\(v\), namely, the outer edge arising from the unique path of
\(\mathcal P\) that ends at \(v\). Then
\[
\sigma_f(v)=4m+2t+\lambda(v).
\]
The values
\[
\lambda(v),\qquad v\in V(G),
\]
are pairwise distinct. Indeed, different endpoint occurrences on
nontrivial paths correspond to different end edges, while the two ends of
a path of length one receive the two distinct outer labels \(m+2i-1\)
and \(m+2i\). Since different original edges use disjoint pairs of outer
labels, no two original vertices receive the same value of \(\lambda\).
It follows that the sums induced by \(f\) at the original vertices are
pairwise distinct.

Moreover, every outer label is at least \(m+1\), and hence
\[
\sigma_f(v)
=
4m+2t+\lambda(v)
\ge
5m+2t+1
>
4m
\]
for every original vertex \(v\). Therefore no original-vertex sum equals
a subdivision-vertex sum. Since the subdivision-vertex sums are pairwise
distinct and the original-vertex sums are pairwise distinct, \(f\) is a
strongly antimagic labeling of \(S_2(G)\).
\end{proof}

\section{An Extension to Graphs with Odd Degrees}
\label{sec:odd-degree-extension}

The full strength of Claim~\ref{clm:path-decomposition}, together with
the direct labeling developed in Section~\ref{sec:direct-labeling},
allows the cubic result to be extended to graphs in which every vertex
has odd degree at least three. The restriction on degree-one vertices is
essential for the present construction, because the separation between
the original vertices and the subdivision vertices relies on every
original vertex being internal to at least one path of the decomposition.

\begin{theorem}
\label{thm:odd-degree-extension}
Let \(G\) be a finite simple graph, not necessarily connected, such that
\(d_G(v)\) is odd and \(d_G(v)\ge3\) for every \(v\in V(G)\). Then
\(S_2(G)\) is strongly antimagic.
\end{theorem}

\begin{proof}
Let \(n=|V(G)|\) and \(m=|E(G)|\). By
Claim~\ref{clm:path-decomposition}, let \(\mathcal P\) be a decomposition
of \(E(G)\) into \(n/2\) paths such that every vertex is an end vertex of
exactly one path. Since \(d_G(v)\ge3\) for every \(v\in V(G)\), we have
\(m\ge3n/2\). Hence at least one path in \(\mathcal P\) has length at
least three, and the construction of
Section~\ref{sec:direct-labeling} can be applied without change.

Let \(t\) be the number of paths of length one in \(\mathcal P\), and let
\(q(v)\) denote the number of paths in which \(v\) is an internal vertex.
By Claim~\ref{clm:path-decomposition},
\[
q(v)=\frac{d_G(v)-1}{2}.
\]
As shown in Section~\ref{sec:direct-labeling}, each path on which \(v\)
is internal contributes \(4m+2t\) to \(\sigma_f(v)\), while the unique
path ending at \(v\) contributes an outer label \(\lambda(v)\). Hence
\[
\sigma_f(v)
=
q(v)(4m+2t)+\lambda(v)
=
\frac{d_G(v)-1}{2}(4m+2t)+\lambda(v),
\]
where the values \(\lambda(v)\), \(v\in V(G)\), are pairwise distinct
and satisfy
\[
m+1\le\lambda(v)\le3m.
\]

Therefore vertices of the same degree receive distinct sums. If
\(d_G(u)<d_G(v)\), then both degrees are odd, so
\[
q(v)\ge q(u)+1.
\]
Consequently,
\[
\begin{aligned}
\sigma_f(v)-\sigma_f(u)
&\ge (4m+2t)+(m+1)-3m\\
&=2m+2t+1\\
&>0.
\end{aligned}
\]
Thus the sums at the original vertices are strictly ordered by degree.

Finally, the subdivision-vertex sums are pairwise distinct and at most
\(4m\), as proved in Section~\ref{sec:direct-labeling}. Since
\(d_G(v)\ge3\), we have \(q(v)\ge1\), and therefore
\[
\sigma_f(v)
\ge
(4m+2t)+(m+1)
=
5m+2t+1
>
4m
\]
for every original vertex \(v\). Hence every original vertex has larger
sum than every subdivision vertex, and \(f\) is a strongly antimagic
labeling of \(S_2(G)\).
\end{proof}

\section{Concluding Remarks}
\label{sec:conclusion}

We have proved that the \(2\)-subdivision of every finite simple cubic
graph is strongly antimagic, and have extended this result to finite
simple graphs in which every vertex has odd degree at least three. The
path decomposition in Claim~\ref{clm:path-decomposition} is the common
structural ingredient in these results. In the cubic case, it leads to two
different constructions, while in the odd-degree case it records
precisely how many decomposition paths pass through each vertex and
thereby makes the resulting vertex sums increase with degree.

Although the direct construction in
Section~\ref{sec:direct-labeling} is shorter, the method developed in
Section~\ref{sec:main-results} is retained for a different reason. It does
not require the original labeling to distinguish all vertex sums
immediately. Instead, it first constructs a labeling for which the sums
are sufficiently large and each repeated sum is carefully controlled, and
then uses an orientation after subdivision to remove the remaining
collisions. This separation of the argument into two stages may provide a
useful approach to other antimagic labeling problems in which a direct
construction is difficult to obtain.

The extension in Section~\ref{sec:odd-degree-extension} also indicates
the present limitation of the direct method. A vertex of odd degree \(d\)
is internal to \((d-1)/2\) paths of the decomposition, so the common
internal contribution separates vertices of different odd degrees. This
argument requires \(d\ge3\); a vertex of degree one has no such internal
contribution, and its sum need not be separated from the sums at the
subdivision vertices.

These observations suggest several natural questions. Is \(S_2(G)\)
antimagic for every connected finite simple graph \(G\) with at least one
edge? More strongly, for which graphs \(G\) is \(S_2(G)\) strongly
antimagic? In particular, can the construction be modified to accommodate
vertices of degree one, or more generally vertices of both odd and even
degrees? It would also be interesting to determine whether the ideas used
here can be adapted to uniform subdivisions \(S_\ell(G)\) for
\(\ell\ge3\).

\section*{Declaration of Generative AI and AI-Assisted Technologies in the Manuscript Preparation Process}

During the preparation of this work, the authors used ChatGPT, developed
by OpenAI, to assist with organizing the presentation, examining the
internal consistency of proposed constructions, exploring possible
exceptional cases, and improving the language and readability of the
manuscript. After using this tool, the authors reviewed and edited the
content as needed and take full responsibility for the content of the
published article.


\end{document}